\documentclass[11pt]{amsart}
\usepackage{amsthm}
\theoremstyle{plain}
\usepackage{hyperref}
\usepackage{cleveref}
\usepackage[colorinlistoftodos,bordercolor=orange,backgroundcolor=orange!20,linecolor=orange,textsize=scriptsize]{todonotes}

\usepackage{graphicx, enumerate, url}
\usepackage[margin=1in]{geometry}
\usepackage{amssymb,mathtools}
\usepackage{bbm} 
\usepackage{algorithm}
\usepackage{algpseudocode}
\usepackage{color}
\usepackage{tikz}
\usetikzlibrary{3d,calc}

\theoremstyle{plain}
\newtheorem{theorem}{Theorem}
\newtheorem{proposition}{Proposition}
\newtheorem{lemma}{Lemma}
\newtheorem{corollary}{Corollary}

\theoremstyle{definition}

\theoremstyle{remark}

\DeclareMathOperator{\E}{E}

\DeclareMathOperator{\diag}{diag}
\DeclareMathOperator{\rank}{rank}

\newcommand{\N}{\mathbb{N}}
\newcommand{\C}{\mathbb{C}}

\newcommand{\Z}{\mathbb{Z}}

\newcommand{\x}{\mathbf{x}}

\newcommand{\cA}{\mathcal{A}}

\newcommand{\cL}{\mathcal{L}}

\newcommand{\bA}{\mathbf{A}}

\newcommand{\bB}{\mathbf{B}}

\newcommand{\bC}{\mathbf{C}}

\newcommand{\bbf}{\mathbf{f}}

\newcommand{\bJ}{\mathbf{J}}

\newcommand{\bl}{\mathbf{l}}
\newcommand{\bm}{\mathbf{m}}

\newcommand{\bX}{\mathbf{X}}
\newcommand{\bY}{\mathbf{Y}}

\newcommand{\rL}{\mathrm{L}}
\newcommand{\rM}{\mathrm{M}}
\newcommand{\rP}{\mathrm{P}}

\newcommand{\rW}{\mathrm{W}}

\newcommand{\0}{\mathbf{0}}

\makeatletter
\@namedef{subjclassname@2020}{\textup{2020} Mathematics Subject Classification}
\makeatother

\begin{document}
\title
{Complete invariants for simultaneous similarity}
\author{Klaus Bongartz}
\address{Universit\"at Wuppertal,Wuppertal, Germany, \texttt{klausbongartz@online.de}}
\author{Shmuel Friedland}
\address{
 Department of Mathematics, Statistics and Computer Science,
 University of Illinois, Chicago, IL 60607-7045,
 USA, \texttt{friedlan@uic.edu}}

 \date{May 14,  2026}
  	
\begin{abstract} 
Always dealing with an arbitrary field we consider the variety  $(k^{n\times n})^{p}$ under the action of $GL_{n}$  by simultaneous similarity.  We define discrete and continuous invariants which completely determine the orbits. The discrete invariants induce a disjoint decomposition of the variety into finitely many locally closed $GL_{n}$-stable subsets and  for each of these 
we construct finitely many invariant morphisms to $k$  separating the orbits.  The complicated action of $GL_{n}$ by similarity is reduced to left multiplication of a product of $GL_{l_{i}}$'s on a product of $k^{l_{i}\times m_{i}}$'s.  An analogous result holds for the left-right action of $GL_{m}\times GL_{n}$ on $(k^{m\times n })^{p}$  and more generally for all varieties of finite dimensional modules over some finitely generated algebra.
\end{abstract}
\maketitle

\section{Introduction}\label{sec:intro}
Given a  $G$-variety $X$ - i.e. an algebraic variety $X$ endowed with an algebraic action  of an algebraic group $G$ - one would like to have a morphism $\pi:X \rightarrow Y$ whose non-empty fibres are the orbits. 
 In that case all orbits are closed in $X$ and so such a morphism seldom exists, but it exists generically by a fundamental result of Rosenlicht. This says that there is always a dense open $G$-stable subvariety $X'$ where such a morphism $\pi:X' \rightarrow k^{t}$ exists. Thus the components of $\pi$ are  finitely many invariant functions separating the orbits. Then one  can remove $X'$ from $X$, and start again with the action of $G$ on $X \setminus X'$ which has strictly smaller dimension than $X$. Therefore, this procedure ends after finitely many steps with a decomposition of $X$ into $G$-stable pairwise disjoint locally closed subvarieties $X_{i}$ admitting morphisms $\pi_{i}:X_{i} \rightarrow k^{t_{i}}$, and having as non-empty fibres the orbits. We call this a separating stratification, and we assume neither that the strata $X_{i}$ are irreducible nor that their closure is a union of some other strata.
 
This iterative ''algorithm'' is hard to realize.  Already the first step is difficult, even though among the many different proofs of Rosenlicht's theorem ( see \cite{R0,R,S,Kr,B4,Ke} ) there is a computational one.    To go on one has to find the irreducible components of $X\setminus X'$ which can be difficult, and then we repeat the algorithm  \cite{F0}.  
  
The main aim of the present article is to  describe in one stroke a separating stratification for the variety 
\begin{equation}\label{defVnp}
V(n,p):= (k^{n\times n})^{p}
\end{equation}
endowed with the action of $GL_{n}$  by simultaneous similarity. 
In particular, we obtain a complete set of invariants for simultaneous similarity.
Before we delve into details we formulate the underlying
simple idea:
\begin{lemma}\label{lem1} Let $X$ be a $G$-variety and $Y$ an $H$-variety. Suppose we have a morphism $\varphi:X \rightarrow Y$ such that the inverse image of an $H$-orbit is empty or a $G$-orbit. If $Y$ has a separating stratification $Y_{1},Y_{2},\ldots ,Y_{N}$ with morphisms $\pi_{i}:Y_{i}\rightarrow Z_{i}$  then the $\pi_{i}\circ \varphi:\varphi^{-1}Y_{i} \rightarrow Z_{i}$ for non-empty $\varphi^{-1}Y_{i}$ is  a separating stratification for $X$. 
\end{lemma}
\begin{proof} Since the morphism $\varphi$ is a continuous map the inverse image of any locally closed set is locally closed.  It follows that the non-empty $\varphi^{-1}Y_i$ are a disjoint decomposition of $X$ into locally closed sets which will not be irreducible in general even if the $Y_i$ are so.  Each $\varphi^{-1}Y_i$ is stable under the action of G since $Y_i$ is stable under the action of $H$.  One verifies easily that one gets a separating stratification.
\end{proof}

For any  natural number $N$ and a finite sequence of natural numbers $l_j,m_{j}$ with $j\in [N]:=\{1,\ldots,N\}\subset \N$  we introduce the variety 
$Y(\bl,\bm)=\prod_{j=1}^{N} k^{l_{j}\times m_j}$ endowed with the obvious action of  $H=\prod_{j=1}^{N} GL_{l_{j}}$  from the left on the rows of matrices in $\prod_{j=1}^Nk^{l_j\times m_j}$.  Then $Y$ has a separating stratification based on the Gauss-algorithm as follows from section \ref{subsec:ref}.
 Now we state our main result still without the technicalities involved in the definition of $\varphi$. 
\begin{theorem}\label{thm1} Let $n$ and $p$ be two natural numbers.  Then one can  construct natural numbers $l_j$ and $m_{j}$ for $j\in [n^{2}]$ and a morphism  
\begin{equation}\label{defphi}
\varphi: V(n,p) \rightarrow Y(\bl,\bm), \quad \bl=(l_1,\ldots,l_{n^2}),\bm=(m_1,\ldots,m_{n^2}),
\end{equation}
 such that the preimage of an $H$-orbit is empty or a $GL_{n}$-orbit.
\end{theorem} 

Thus, we obtain a separating stratification for the complicated action by conjugation from the almost trivial action by left-multiplication. The drawbacks are that the target variety $Y$ is very big and growing fast with $n$ and $p$, and that we do not know the image of $\varphi$. In particular there will be many  strata $Y_{i}$ of $Y$ with empty preimage.

Now we describe shortly what is done in the article.  In section \ref{sec:geo1m} we recall in the language of $G$-varieties the results on the reduced row echelon form and the rational normal form.  In section \ref{sec:repth}  we give a  result in representation theory which is essential in proving the main result of our paper, that is stated in the central section \ref{sec:prfthm1}.    In this section we give a complete set of invariants for simultaneous similarity and prove Theorem \ref{thm1}.
We apply this result in section \ref{sec:oa} to obtain a separating stratification for general varieties of modules, including tuples of $m\times n$- matrices under the left-right multiplication of $GL_{m}\times GL_{n}$.
The final section \ref{sec: crr} makes some comments on related articles  and problems.

Implicitly, the main result is already  published by the authors in the  articles \cite{F1,F2} from 1983 and \cite{B3} from 1995. The first article of the second author constructs  a stratification into pieces with invariant functions which separate the orbits only up to finitely many choices. This construction was then refined to a separating stratification by the first author  after observing our Theorem \ref{thm3}.  However the third article went unnoticed even by the second author. 

To remedy this situation we wrote together this article which has a strong didactical flavour. To make it easier to  read for any mathematician we  have already explained above the simple and beautiful idea of the second author  before delving into technicalities. To make it easier to read for people from matrix theory we give in subsections \ref{subsec:geonf} and \ref{subsec:condef} some definitions and facts from elementary algebraic geometry as presented in the first two chapters of \cite{Hu95}.   In subsections \ref{subsec:ref} and \ref{subsec:rnf} we formulate  two well-known results about matrices in that geometric language.  For representation theorists  we list in subsection \ref{subsec:prelm} some basic facts from matrix theory like the Cauchy-Binet formula and the use of Kronecker products for matrix equations as given in \cite{Fribk}. 
\section{ Some geometry and the case of one matrix}\label{sec:geo1m}
\subsection{Geometric quotients and normal forms}\label{subsec:geonf}
Let $X$ be a $G$-variety and $\pi:X \rightarrow Y$ be a morphism with the orbits as non-empty fibres. Then one 
finds a subvariety $Y'$ of $Y$ such that the induced morphism $\varphi:X'=\pi^{-1}Y'  \rightarrow Y'$ is even a geometric quotient ( see \cite[5.3]{B4}   ).   This means that $\varphi$ maps open sets to open sets, has the orbits as fibres and it induces for each open subset $U$ of $Y'$ an isomorphism $\mathcal{O}_{X'}( \varphi^{-1}(U))^{G} \simeq \mathcal{O}_{Y'}(U) $. We say that a geometric quotient $\varphi:X \rightarrow Y$  admits   normal  forms if there is a constructible subset $N$ of $X$ hitting each orbit in exactly one point.

In subsections  \ref{subsec:ref} and \ref{subsec:rnf}  we present for the two classical actions of linear algebra separating stratifications $\pi_{i}:X_{i}\rightarrow Y_{i}$  which are geometric quotients admitting affine subspaces as normal forms.  So here the situation is very nice. 
In  section \ref{subsec:wild} we will mention bad behaviour for pairs of $2\times 2$ and $3\times 3$-matrices. 

\subsection{Some conventions, definitions, facts}\label{subsec:condef} 
Recall that in this article $k$ is an arbitrary field. 
We refer to the first pages of Humphreys book \cite{Hu95} for all what follows in this subsection. 
A subset of a topological space is locally closed if it is the intersection of an open and a closed subset and irreducible if it is not the union of two proper closed subsets. On $k^{n}$ we have the Zariski-topology where the closed sets are defined by the vanishing of a certain set of polynomials in $k[X_{1},X_{2},\ldots ,X_{n}]$. Such a set is an affine variety. An open subset of an affine variety is a quasi-affine variety and these are the only varieties we are dealing with from now on. Any locally closed subset of a variety is itself a variety.
A finite union of locally closed subsets is called constructible.
 
Morphisms between two varieties are functions that can locally  be  defined as the quotient of two polynomials where the denominator  does not vanish.  An isomorphism is a bijective morphism where the inverse is also a morphism. 

Any morphism $f:X \rightarrow Y$ is continuous for the Zariski-topologies. It follows that the image of an irreducible set is irreducible and the preimage of a locally closed set again locally closed.  A  variety $X$ is smooth if it is irreducible and if the dimension of the variety coincides  with the dimension of the tangent space $T_{x}X$ in each point $x$, and X is rational if an open dense subset of $X$ is isomorphic to an open subset of some $k^{n}$.

The product $X_{1}\times X_{2}$ of two varieties is just the set-theoretic product and this product is irreducible, resp. rational, resp. smooth if both factors are so. The product of two locally closed subsets $L_{i}$ in $X_{i}$ is again locally closed. 

An algebraic group $G$ is an algebraic variety endowed with a group structure such that the multiplication and the inversion are morphisms of algebraic variety. A $G$-variety is a variety $X$ equipped with an action $G\times X \rightarrow X$ which is a morphism.  
The orbit of each $x \in X$,  denoted by orb$(x)=\{y\in X: y=gx, g\in G\}$,  is then a locally closed subset. We consider only $G=GL_{n}$ or some products of such groups acting on varieties $X$ of matrices or products thereof with action by conjugation or by left multiplication.

\subsection{Reduced row echelon forms}\label{subsec:ref}
First we treat the action of $G=GL_{l}$ on $X= k^{l\times m}$ by left multiplication. Then the orbit of $A$ contains always exactly one reduced echelon matrix $E(A)$ and Gauss-elimination gives an effective algorithm to compute it. Since some divisions are necessary in this process  the map $A \mapsto E(A)$ is not a morphism, but its  restriction to appropriate locally closed subsets is.

Any subset $F$ of $[\min(l,m)]$ has cardinality $r \leq \min(l,m)$ and ordering  its elements $f(1)<f(2)< \ldots <f(r)$ we get a strictly increasing function $f:[r] \rightarrow [\min(l,m)]$.
The set$X_{F}=k^{l\times m }_{F}$ consists in those matrices $A$ with rank $r$ where  $f(i)$ is the smallest number $j$ such that the span of the first $j$ columns has dimension $i$. This set is stable under $G$ and it is locally closed because it is defined by the vanishing resp. non-vanishing of appropriate minors of $A$. Furthermore $X$ is the disjoint union of the various $ X_{F}$. Here $X_{\emptyset}$ contains the $l\times m$ zero-matrix.

For each $F$ as above  with associated function $f$ the set $E_{F}$ contains all matrices $A \in X_{F}$ such that only the first $r$ rows are not $0$, the first non-zero entry in the $i^{th}$ row is $A_{if(i)}=1$ and all other entries in the column $f(i)$  are $0$. So $E_{F}$ is just an affine subspace of $k^{l\times m}$ having the matrix with exactly $r$ non-zero entries as a base point. The matrices in $E_{F}$ are called reduced echelon forms. The zero-matrix is the reduced echelon form in $X_{\emptyset}$. 
 \begin{proposition}\label{prop1} Each $  X_{F}$ is an irreducible smooth rational variety. The map $A \mapsto E(A)$ is the geometric quotient  
 and it admits the  affine subspace $E_{F}$ as normal forms.
 \end{proposition}  
 \begin{proof}We just follow the steps taken in the Gauss algorithm. To see that $X_{F}$ is smooth and rational and that $E$ is a morphism we consider the open subset $U$ of   all $A$ with $A_{1f(1)}\neq 0$. Here we can forget about the first $f(1) -1$ columns and suppose $f(1)=1$. Denote by $Y$ the subset of $U$ containing all $B$ with $B_{ij}=0$ if $j<f(2)$ and  $i\geq 2$. Then the Gauss-algorithm gives a morphism $h:U \rightarrow k^{l-1}\times Y$ mapping a matrix $A \in U$ to the tuple $(A_{21},A_{31},\ldots,A_{l1},B)$. Here $B$ is defined by $B_{1\bullet}=A_{1\bullet}$ and $B_{i\bullet}=A_{i\bullet}-\frac{A_{i1}}{A_{11}}A_{1\bullet }$.  The submatrix $M$ of $A$ formed by the first $f(2)-1$ columns has rank 1 and so all rows of $M$ are proportional to the first row. Looking at the first column of $B$ one sees that it belongs to $Y$. The inverse of  $h$ is also a  morphism. 
 All the matrices  $B'$ consisting of  the last $l-1$ rows of $B$ form some $X'_{F'}$ in $X'=k^{(l-1)\times (m-f(2)+1)}$ which is a smooth rational variety by induction on $r$. The same holds for $U$ which is isomorphic to the product $k^{l-1}\times k^{\ast}\times k^{m-1}\times X'_{F'}$.
 
 Again by induction the map $B' \mapsto E(B')$ is a morphism and $E(B')$ determines the last $l-1$ rows of $E(B)$. By adding appropriate multiples of these rows to $B_{1\bullet} $ and multiplying it with $A_{11}^{-1}$ we obtain $E(B)=E(A)$. We have shown that $E$ is a morphism on $U$. But the $G$-translates of $U$ cover $X_{F}$ and $E$ is $G$-invariant whence a morphism  in an open neighborhood of each point. Furthermore $X_{F}$ is smooth.
 
 Now $E:X_{F}\rightarrow E_{F}$ is a morphism with the orbits as fibres and the inclusion $i:E_{F} \rightarrow X_{F}$ satisfies $\pi \circ  i=id$. This implies that $\pi$ is the geometric quotient. Finally $X_{F}$ is irreducible as the image of the irreducible variety $G\times E_{F}$.\end{proof}
 
 We have described a separating stratification $\pi_{F}:X_{F} \rightarrow E_{F}$ with affine spaces  $E_{F}$ as quotients. These appear also as 
parts of the following more global construction using  Grassmannians. For any $r \leq min(l,m)$ one looks at the locally closed subvariety $X(r)$ formed by all matrices $A\in k^{l\times m}$ of rank r.  The space $Row(A)$ spanned by the rows of $A$ is then a subpace of $k^{m}$ of dimension $r$  and the map $A \mapsto Row(A)$ is the geometric quotient $\pi:X(r) \rightarrow Gr(r,m)$ with the Grassmannian of $r$-dimensional subspaces in $k^{m}$ as the quotient. The affine spaces $E_{F}$  to the sets $F$ with $r$ elements are isomorphic to the Schubert-cells of $Gr(r,m)$ and the $\pi_{F}$ are  pull-backs of $\pi$. In the following we stick to the affine pieces, but we will use the fact that two matrices $A$ and $B$ lie in the same orbit iff $Row(A)=Row(B)$ iff $E(A)=E(B)$.
\subsection{Rational normal forms}\label{subsec:rnf}
We  now look at the action of $G=GL_{n}$ on $X=k^{n \times n}$ by similarity.  Again we find a separating stratification consisting of irreducible rational smooth subvarieties having geometric quotients that admit normal forms living in affine subspaces. But this time the proofs are much harder and we refer to the literature. 
 
A rational normal form $R \in  k^{n\times n}$, also known as the Frobenius normal form,  is a block-diagonal matrix
 \begin{equation}\label{RNF}
\left [
\begin{array}{cccc}
B(P_{1})&0& ...&0\\
0&B(P_{2})& ... & 0\\
...&...&...&...\\
0&0&...&B(P_{r})
\end{array}
\right ]
\end{equation}
with companion-matrices $B(P_{1}),B(P_{2}),\ldots ,B(P_{r})$ on the main diagonal, such that all polynomials $P_{i}$ are normalized with coefficients in $k$ and  $P_{i+1}$ divides 
 $P_i$ for all $i$ with  $1\leq i \leq r-1$. We write $R=R(P_{1},P_{2},\ldots ,P_{r})$ for such a matrix and call it an RNF.

The main result about the RNF says: Each matrix $A \in k^{n\times n}$ is similar  to exactly one rational normal form $R=R(A)=R(P_{1},P_{2},\ldots ,P_{r})$ which can be obtained from the coefficients of $A$ by finitely many rational operations.

 Again the map $A \mapsto R(A)$ is not a morphism but its restriction to appropriate subvarieties given by partitions of $n$ is. To each partition $p=(p_{1},p_{2},   \ldots ,p_{r})$ we collect all rational normal forms $R=R(P_{1},P_{2},\ldots ,P_{r})$ with $P_{i}$ of degree $p_{i}$  for all $i$ to the subset $R(p)$ and  all $A$ with $R(A)$ in $R(p)$ to $S(p)$.  These $S(p)$ are just the sheets for the considered action i.e. the irreducible components of the finitely many locally closed subsets where the orbit dimension is the same (\cite{Bo}). 
 
 \begin{theorem}\label{thm2}Each $ S(p)$ is an irreducible rational smooth subvariety of $k^{n\times n}$. The map $A \mapsto R(A)$ is the geometric quotient 
 $R:S(p) \rightarrow R(p)$.  Taking a slight variant of companion matrices on gets even an affine subspace $A(p)$ as normal forms.
 
 \end{theorem}
 
 The proof is given in  \cite{B1,B7}. The existence of $A(p)$ answers a question asked by Kontsevich in  \cite[pages 127 and 613]{Ar}.

\section{A fact from representation theory}\label{sec:repth} 
In what follows we denote by $\cA$ an associative algebra.
We need a test deciding when two $\cA$-modules $M$ and $N$ of finite dimension are isomorphic.  This test is in terms of dimensions of various  homomorphism spaces and we abbreviate with $[X,Y]$ the dimension of $Hom_{\cA}(X,Y)$.

Using the existence of almost split sequences M. Auslander proved in \cite{Aus}
that two modules $M$ and $N$ as above are isomorphic if and only if
$[M,X]=[N,X]$  holds for all finite-dimensional modules $X$.  Here we show
that $M$ and $N$ are already isomorphic if $[M,X]=[N,X]$ holds only for all
submodules of $M$ and of $N$,  and this is essential to derive our main
result.  Our proof  is elementary,  and it extends to $k$-linear abelian
categories with finite dimensional homomorphism spaces, where almost
split sequences do not exist, e.g. , to the category of coherent sheaves
over any projective variety.  We use properties of the radical of the category $mod\,\cA$ of finite-dimensional modules ( see e.g. \cite[appendix A.3]{As} ), the additivity in both variables of the Hom-functor and its exactness properties.

\begin{theorem}\label{thm3} Let $k$ be an arbitrary field and $\cA$ an associative $k$-algebra with two  finite dimensional modules $M$ and $N$. Then the following conditions are equivalent:
\begin{enumerate}[(a)]
\item $M$ and $N$ are isomorphic.
\item $[ M,X]=[N,X]$ holds for each $X$ which is a submodule of $M$ or of $N$.
\item $[X,M]=[X,N]$ hods for each $X$ which is a submodule of $M$ or of $N$.
\end{enumerate}
 
\end{theorem}
\begin{proof} That (a) implies (b) is trivial.  For the other direction we take a basis $f_{1},f_{2},\ldots ,f_{r}$ of $Hom_{\cA}(M,N)$ and we consider the homomorphism $f:M^{r} \rightarrow N$ with 
$f(m_{1},m_{2},\ldots ,m_{r})= \sum f_{i}m_{i}$ and its image $I$ in $N$. Any homomorphism $g:M \rightarrow N$ is a linear combination of the $f_{i}$'s and therefore its image
is contained in $I$. Thus the inclusion $\epsilon:I \rightarrow N$ induces an isomorphism $Hom_{\cA}(M,I) \rightarrow Hom_{\cA}(M,N)$. Of course $\epsilon$ induces also an injection $Hom_{\cA}(N,I) \rightarrow Hom_{\cA}(N,N)$ which is bijective because of $[N,N]=[M,N]=[M,I]=[N,I]$. Thus the identity $id_{N}$ factors through $\epsilon$ and so $f$  is surjective. Symmetrically we find also a surjection $g:N^{s} \rightarrow M$.

Consider  decompositions $M=\bigoplus_{i=1}^{p} U_{i}$ and  $N=\bigoplus_{j=1}^{q} V_{j}$  into indecomposables. Suppose first that none of the $U_{i}$ is isomorphic to one of the $V_{j}$. We have an  epimorphism $h:M^{rs} \rightarrow N^{s} \rightarrow M$ which we can write as a product of huge  matrices of homomorphisms between non-isomorphic indecomposables so that they all belong to the radical of the category $mod\,\cA$.  Thus $h$ shows   that $JM=M$ where $J$ is the usual nilpotent Jacobson-radical of the finite-dimensional algebra $End\,M$. It follows that $M=0$ whence $0=[M,M]=[N,N]$ and $N=0$.

In the general case we have  $M=K\oplus M'$ and $N=L\oplus N'$  in such a way that $K$ and $L$ are isomorphic whereas no indecomposable direct summand of $M'$ is isomorphic to one of $N'$. Any $X'$ which is a submodule of $M'$ or of $N'$ is a fortiori a submodule of $M$ or of $N$. Thus we have 
$[K,X']+[M',X']=[M,X']=[N,X']=[L,X']+[N',X']$  whence $[M',X']=[N',X']$. By the first case $M'=N'=0$ and so $M$ and $L$ are isomorphic.

It remains to show that (c) implies (a). This is not dual to (b) implies (a). Nevertheless we  look at the transposed morphism  $g:M \rightarrow N^{r}$ with components the $f_{i}$ from the proof above. For its kernel $K$ one has $0=[K,N]=[K,M]$ and so $g$ is a mono. Symmetrically one obtains a mono $M \rightarrow N^{s}$ and then one applies the arguments dual to the ones used in the proof above. 
\end{proof}

We consider later on  the free associative algebra ${\mathcal A}=k\langle X_{1},X_{2},\ldots ,X_{p}\rangle$ in $p$ non-commuting variables. Any $p$-tuple ${\bf A}=(A_{1},A_{2},\ldots ,A_{p})$ of $n\times n$-matrices defines the $n$-dimensional module $ M_{\bf A}$ where $X_{i}$ acts through $A_{i}$. A homomorphism from $ M_{\bf A}$  to   $ M_{\bf B}$ is a matrix $Z \in k^{n\times n}$ with $ZA_{i}=B_{i}Z$ for all $i$. Therefore
 $ M_{\bf A}$  and   $ M_{\bf B}$ are isomorphic iff the two tuples $\bf A$ and $\bf B$ are simultaneous similar. Any $n$- dimensional module over  $\mathcal{A}$ is isomorphic to some $M_{\bf A}$.

\begin{corollary}\label{cor1} For ${\mathcal A}$ two modules $M$ and $N$ of dimension $n$ are isomorphic iff $[M,X]=[N,X]$ holds for all modules of dimension $n$.
\end{corollary}
\begin{proof} Let $S$ be the one-dimensional module where all $X_{i}$ act as $0$. Then we have $[M,S^{n}]=[N,S^{n}]$ whence $[M,S]=[N,S]$.
   Given a submodule $X$ of $M$ or $N$ of dimension $d$ we get   $[M,X\oplus S^{n-d}]=[N,X\oplus S^{n-d}]$ and so $[M,X]=[N,X]$. By the theorem $M$ and $N$ are isomorphic.
\end{proof}

 \section{The definition of $\varphi$ and the proof of Theorem \ref{thm1}}\label{sec:prfthm1}
\subsection{Preliminary results from matrix theory}\label{subsec:prelm}
In this subsection we recall some well known results in matrix theory that we will use, see \cite{Fribk}.    Let $A=[a_{ij}]\in k^{m\times n}$, where $a_{ij}$ are the entries of $A$.  Denote by $I_n\in k^{n\times n}$ the identity matrix,  and $A^\top\in k^{n\times m}$ the transpose of $A$.  Let $[n]_l$ be the set of all subsets of $[n]$ of cardinality $l$.  Thus, $\alpha\in[n]_l$ is of the form $\alpha=(\alpha_1,\ldots,\alpha_l), 1\le \alpha_1<\cdots<\alpha_l\le n$.  Assume that $l\in[\min(m,n)]$ and $\alpha\in[m]_l,\beta\in[n]_l$.  Then $A[\alpha,\beta]=[a_{\alpha_i\beta_j}]\in k^{l\times l}$ is a submatrix of $A$ of order $l$ formed by the rows $\alpha$ and columns $\beta$,  and $\det A[\alpha,\beta]$ is an $l$-minor 
of $A$.  Let $B\in k^{n\times q}$, and $C=AB\in k^{m\times q}$.  Recall that $\rank C\le \min(m,n,q)$.  Hence any $r$-minor of $C$ is zero for $r>\min(m,n,q)$.     The Cauchy-Binet formula is the following identity for an $r$-minor of $C$, where $r\in [\min(m,n,q)]$  \cite[Propostion 5.2.7]{Fribk}:
\begin{equation}\label{CB}
\det C[\alpha,\beta]=\sum_{\gamma\in [n]_r} \det A[\alpha,\gamma]\det B[\gamma, \beta], \quad \alpha\in [m]_r, \beta\in[q]_r.
\end{equation}

Assume now that $B=[b_{qr}]\in k^{s\times t}$.  The Kronecker tensor product $C=A\otimes B$ is an $(ms)\times (nt)$ matrix whose entries are $c_{(i,q)(j,r)}=a_{ij}b_{qr}$ arranged in a lexicographical order for a row $(i,q)$ and a column $(j,r)$.  It is represented as the following block matrix \cite[Definition 5.1.7]{Fribk}:
\begin{equation}\label{Kropr}
A\otimes B=[a_{ij}B]:=\begin{bmatrix}a_{11}B&\cdots&a_{1n}B\\\vdots&\vdots&\vdots&\\a_{m1}B&\cdots&a_{mn}B\end{bmatrix}.
\end{equation}
That is, the Kronecker tensor product is a representation of $k^{m\times n}\otimes k^{s\times t}$ as $k^{(ms)\times (nt)}$.  
Next, we  represent the matrix multiplications $X\mapsto BX$ and $X\mapsto XA$,  where $X$ varies in $k^{n\times n}$ and $A,B$ are fixed matrices in $k^{n\times n}$.
Write $X$ as $[\x_1\,\x_2\cdots\x_n]$, where $\x_1,\ldots,\x_n$ are the columns of $X$.  View $X$ as a vector $\hat X=\begin{bmatrix}\x_1\\\vdots\\\x_n\end{bmatrix}\in k^{n^2}$.    It is quite straightforward to see $X\mapsto BX$ is given by $\hat X\mapsto (I_n\otimes B)\hat X$.   Slightly more work is needed to show that $X\mapsto XA$ is given by 
$\hat X\mapsto (A^\top\otimes  I_n)\hat X$.  Finally it is straightforward to show using \eqref{Kropr} that 
\begin{equation}\label{tpf}
(A\otimes B)(C\otimes D)=(AC)\otimes (BD).
\end{equation}
 \subsection{The matrices $\rL(\bA,\bC)$}\label{subsec:LAC}
 We view $\bA=(A_1,\ldots,A_p)\in V(n,p)$ as an $pn\times n$ matrix: $\begin{bmatrix}A_1\\\vdots\\A_p\end{bmatrix}$.
For $\bA,\bB\in V(n,p)$ we denote $\bA\approx \bB$ if $\bA$ and $\bB$ are (simultaneously) similar, i.e.  $\bB=T\bA T^{-1}, T\in GL_n$.
For a given  $\bA,\bC\in V(n,p)$  we associate a linear transformation \cite[Section 1]{F1}:
\begin{equation}\label{defLAB}
\begin{aligned}
&\cL(A,C):k^{n\times n}\to k^{n\times n} \quad  \cL(A,C)(Z)=ZA-CZ,\quad A,C,Z\in k^{n\times n},\\
&\cL(\bA,\bC): k^{n\times n}\to (k^{n\times n})^p:  \cL(\bA,\bC)(Z)=
(ZA_1-ZC_1,\ldots,ZA_p-C_p Z)=Z\bA-\bC Z.
\end{aligned}
\end{equation}
The matrix representations of $\cL(A,C)$ and $\cL(\bA,\bC)$,  see above,  are $\rL(A,C)$ and $\rL(\bA,\bC)$:
\begin{equation}\label{matrepLAB}
\begin{aligned}
&\rL(A,C)=A^\top\otimes I_n -I_n\otimes C,\\
&\rL(\bA,\bC)=\begin{bmatrix} A_1^\top \otimes I_n-I_n\otimes C_1\\\vdots\\ A_p^\top \otimes I_n-I_n\otimes C_p\end{bmatrix}
\end{aligned}
\end{equation}

The second line in \eqref{defLAB}  shows that $Z$ is in the nullspace  of $\mathcal{L}({\bf A},{\bf C})$ if and only if it is in  $Hom_{\mathcal{A}}(M_{\bf A},M_{\bf C}) $.   Thus we obtain   the important equation 
\begin{equation}\label{dimfor}
n^{2}=[M_{\bA},M_{\bC}]+ \rank \rL(\bA,\bC).
\end{equation}
 
 Assuming that  ${\bf B}=S{\bf A}S^{-1}$ we deduce as in [17,Eq. 1.4]  the relation
 \begin{equation}\label{eq1.4}
 \rL(\bB,\bY)=\diag(T\otimes I_n,\ldots,T\otimes I_n)\rL(\bA,\bY)(T^{-1}\otimes I_n), \quad T=(S^\top)^{-1}
 \end{equation}

\subsection{Invariants}\label{subsec:inv}
To find the invariants and to  define the morphism $\varphi$ we work for a while not over $k$ but over the commutative polynomial ring 
 $R=k[y_{s,t,l}]$ in $pn^{2}$ variables  with the matrices $Y_{l}=[y_{s,t,l}] \in R^{n\times n}$ and the $p$-tuple ${\bf Y}=(Y_{1}, \ldots ,Y_{p})$. 
 
Define
\begin{equation}\label{defLNnpr}
L(n,p,r)={pn^2\choose r}\times{n^2\choose r}, \quad N(n,p,r)=\sum_{l=0}^r {pn^2+l-1 \choose l}, \quad r\in[n^2].
\end{equation}
Let $\rP(pn^2,r)$ be the subspace of homogeneous  polynomials of degree $r\in[n^2]\cup\{0\}$
in $pn^2$ variables in the entries of $\bY=(Y_1,\ldots,Y_p)$.  Choose a basis in this space consisting of monomials in the entries of $\bY\in V(n,p)$.  Each monomial is 
\begin{equation*}
\bY^{\bJ}=\prod_{s=t=l=1}^{n,n,p} y_{st,l}^{j_{st,l}},  \quad \bJ=(J_1,\ldots,J_p)\in (\Z_+^{n\times n})^p, \sum_{s=t=l=1}^{n,n,p} j_{st,l}=r.
\end{equation*}
Here $\Z_+$ is the set of nonnegative integers.
The dimension of this subspace is ${pn^2+r-1\choose r}$ \cite{F1}.     Order this basis in a lexicographical order.
We choose a basis in the space of polynomials of at degree $\le r$ in the entries of $\bY$, denoted as $\hat\rP(pn^2,r)$, that consists of monomials  of degree $\le r$:
first $1$,  then the basis of $\rP(pn^2,1)$,$\dots$, the basis of $\rP(pn^2,r)$.
Hence, $\dim \hat\rP(pn^2,r)=N(n,p,r)$.
A polynomial $f\in \hat\rP(pn^2,r)$ is represented in the above basis  by a row $\bbf^\top, \bbf\in k^{N(p,n,r)}$.  

Let $\bA\in V(n,p)$.  For each $r\in[n^2]$ we associate with $\bA$ a matrix  $F(r,\bA)\in k^{L(n,p,r)\times N(n,p,r)}$ obtained as follows.   Consider all $r$-minors of $\rL(\bA,\bY)$, where $\bY=(Y_1,\ldots,Y_p),Y_l=[y_{ij,l}], i,j\in[n], l\in[p]$, and all entries of $\bY$ are variables.  The number of such minors is $L(n,p,r)$ \cite{F1}.
Each minor is a polynomial of degree at most $r$ in $pn^2$ variables.
For $\alpha\subseteq [pn^2],\beta\subseteq [n^2], |\alpha|=|\beta|=r$ let $\bbf(\alpha,\beta,,{\bf A})^\top$ be the vector that represents the polynomial $\det \rL(\bA,\bY)[\alpha,\beta]$ in the above basis of $\hat P(pn^2,r)$.   Arrange these rows in a lexicographical order in $(\alpha,\beta)$ into the matrix $F(r,\bA)\in k^{L(n,p,r)\times N(n,p,r)}$ matrix.    Let  $\rW(r,\bA)$ be the row space of $F(r,\bA)$, and  $\rho(r,\bA)=\rank F(r,\bA)$.
Denote by $\widetilde W(r,\bA)$  the polynomial subspace  spanned by all $r$-minors of $\rL(\bA,\bY)$.    Observe that $\dim \widetilde W(r,\bA)=\rho(r,\bA)$.  Bring $F(r,\bA)$ to its reduced row echelon form $E(r,\bA)$.  
Observe that the rows $1,\ldots,\rho(r,\bA)$ of $E(r,\bA)$ give a unique basis of $\widetilde W(r,\bA)$. 
\begin{theorem}\label{funthm}
Let $\bA,\bB\in V(n,p)$.  Then $\bA\approx \bB$ if and only if $E(r,\bA)=E(r,\bB)$ for all $r\in[n^2]$.
\end{theorem}
\begin{proof}
Assume first that $\bB=S\bA S^{-1}$.  Then \eqref{eq1.4} holds.
Apply Cauchy-Binet identity \eqref{CB} for the right-hand side of the first identity in \eqref{eq1.4}: 
\begin{equation*}
\begin{aligned}
&\det \rL(\bB,\bY)[\alpha,\beta]=\sum_{\gamma,\delta} \det T_1[\alpha,\gamma]\det \rL(\bA, \bY)[\gamma,\delta] \det T_2[\delta,\beta],\\
&T_1=\diag(T\otimes I_n,\ldots,T\otimes I_n), T_2=T^{-1}\otimes I_n, |\alpha|=|\beta|=|\gamma|=|\delta|=r.
\end{aligned}
\end{equation*}
Hence,  the $r$-minor $\det \rL(\bB,\bY)[\alpha|\beta]$ is a linear combination of $r$-minors of $\rL(\bA,\bY)$.  Therefore, $\widetilde W(r,\bB)\subseteq \widetilde W(r,\bA)$. As $\bA=S^{-1}\bB S$ we deduce that $\widetilde W(r,\bB)= \widetilde W(r,\bA)$.   Thus,
$F(r,\bB)$ is row equivalent to $F(r,\bA)$, and $\E(r,\bA)=E(r,\bB)$ for $r\in [n^2]$.

Assume now that $E(r,\bA)=E(r,\bB)$ for $r\in [n^2]$.   Hence,   for $r\in[n^2]$ one has the equality $\widetilde{W}(r,\bA)=\widetilde{W}(r,\bB)$.  
Corollary  \ref{cor1} states that $\bA$ and $\bB$ are simultaneous similar if and only if
\begin{equation}\label{Boncon}
[\rM_{\bA},\rM_{\bX}]=[\rM_{\bB},\rM_{\bX}] \textrm{ for all } \bX\in (k^{n\times n})^p.
\end{equation}
In view of \eqref{dimfor} we obtain that  $\bA$ and $\bB$ are simultaneous similar if and only if
\begin{equation}\label{Boncon1}
\rank \rL(\bA,\bX)=\rank \rL(\bB,\bX) ) \textrm{ for all } \bX\in (k^{n\times n})^p.
\end{equation}
(Compare that with \cite[Theorem 1.1]{D}.)  The $r$-minors of $\rL(\bA,\bY)$ span $\widetilde{W}(r,\bA)$, hence
\begin{equation*}
r> \rank\rL(\bA,\bX)\iff f(\bX)=0 \quad \forall f\in \widetilde{W}(r,\bA).
\end{equation*}
Similarly,
\begin{equation*}
r> \rank\rL(\bB,\bX)\iff f(\bX)=0 \quad \forall f\in \widetilde{W}(r,\bB).
\end{equation*}
Since  $\widetilde{W}(r,\bA)=\widetilde{W}(r,\bB)$ we deduce that \eqref{Boncon1} holds.
\end{proof}

We claim that for each $\bA\in V(n,p)$ the set of $r$-minors of $\rL(\bA,\bY)$ contains a nontrivial polynomial of degree $r$.
Observe that if $\bY\in V(n,p)$, and one of $Y_l\in k^{n\times n}$ is invertible, then there exists at least one $r$-minor of $\rL(\bA,\bY)$ of degree $r$ for each $r\in[n^2]$.
Indeed,  the homogeneous part of degree $r$ of the minor $\det\rL(\bA,\bY)[\alpha,\beta]$ is $\det \rL(\0,\bY)[\alpha,\beta],|\alpha|=|\beta|=r$.  
Note that the equation $Z0-Y_lZ=0$ has a unique solution $Z=0$.  Hence, $\rank\rL(\0,\bY)=n^2$.  Hence for each $r\in[n^2]$
there exists a minor $\det\rL(\0,\bY)[\alpha,\beta]\ne 0$, which yields that the degree of this minor is $r$.  
\subsection{ A proof of Theorem \ref{thm1} and the separating stratification of $(k^{n\times n})^{p}$}\label{subsec:prfthm1}
Let 
\begin{equation}\label{defblbm}
\begin{aligned}
&l_j=L(n,p,j), m_j=N(n,p,j), j\in[n^2], \quad \bl=(l_1,\ldots,l_{n^2}), \bm=(m_1,\ldots,m_{n^2}),\\
&Y(\bl,\bm)=\prod_{j=1}^{n^2} k^{l_j\times m_j}.
\end{aligned}
\end{equation}
Define the map
\begin{equation}\label{defphi}
\begin{aligned}
&\varphi: V(n,p)\rightarrow Y(\bl,\bm),\\
&\varphi(\bX)=(F(1,\bX),\ldots,F(n^2,\bX)).
\end{aligned}
\end{equation}
Clearly, $\varphi$ is a polynomial map in the entries of $\bX$, hence it is a morphism.

Theorem \ref{funthm} implies that $\varphi$ has the properties claimed in Theorem \ref{thm1}. The functions ${\bf A} \mapsto E(r,{\bf A})$ are not morphisms but their restrictions to the preimages of the strata in $Y({\bf l,m})$ are morphisms.

 By section \ref{subsec:ref} each $X=k^{l_{i}\times m_{i}}$ has the separating stratification $\pi_{F}:X_{F}\rightarrow  E_{F}$ for the $GL_{l_{i}}$-action with $F \subseteq [min(l_{i},m_{i})]$. For ${\bf F}=(F_{1},F_{2},\ldots ,F_{n^{2}})$ with $F_{i} \subseteq [min(l_{i},m_{i})]$  the $$S_{\bf F}=\prod_{i \in [n^{2}]} k^{l_{i}\times m_{i}}_{F_{i}}$$ are then the strata for the induced separating stratification of  $\prod_{i \in [n^{2}]} k^{l_{i}\times m_{i}}$.  Thus there are very many strata. Their non-empty inverse images $\phi^{-1}(S_{\bf F})$ are the wanted strata of $V(n,p)$. Their description seems to be a hard problem.
\section{Separating stratifications for other actions}\label{sec:oa} 
\subsection{Varieties of finite dimensional modules}\label{subsec:vfdm}
 Any  algebra $\cA$ with $p$ generators is isomorphic to $k\langle X_{1},X_{2},\dots ,X_{p}\rangle/I$ for some ideal I. Then the variety $mod^{n}\,\cA$ contains the tuples $m=(m_{1},m_{2},\ldots ,m_{p})$ with $P(m_{1},m_{2}, \ldots ,m_{p})=0$ for some generators  $P$ of $I$. Thus  $mod^{n}\,\cA$ is a closed $GL_{n}$ stable subvariety of $V(n,p)$ from which it inherits a separating stratification. 
 
This remark applies in particular to the path algebra of any finite quiver and its quotients.  Our next subsection is just the special case of a quiver with two points
$x, y$ and $p$ arrows from $x$ to $y $,  and we explain the outcome in the language
of matrices.
\subsection{Rectangular matrices}\label{subsec:rm} 
For any natural numbers  $m,n,p$  we define a morphism 
\begin{equation}\label{fwmnp}
\begin{aligned}
&W(m,n,p):=(k^{m\times n})^p,\\
&f:W(m,n,p) \rightarrow V(m+n,p+2),\\
&f(A_1,\ldots,A_p)=(B_1,\ldots,B_{p+2}),\\
&B_{i}=\left [
\begin{array}{cc}
0&A_{i}\\
0& 0\\
\end{array}
\right ],i\in[p], \quad B_{p+1}=\left [
\begin{array}{cc}
I_{m}&0\\
0& 0\\
\end{array}
\right ];\quad B_{p+2}=\left [
\begin{array}{cc}
0&0\\
0& I_{n}\\
\end{array}
\right ].
\end{aligned}
 \end{equation} 
 
  A trivial calculation with block-matrices  starting with $B_{p+1}$ and $B_{p+2}$ shows that two elements $\bA,\bA'$ in   $W(m,n,p)$ are in the same $GL_{m}\times GL_{n}$-orbit iff $f\bA$ and $f\bA'$ are in the same $GL_{n+m}$-orbit. Thus $W(m,n,p)$ inherits a separating stratification from $V(m+n,p+2)$.
\subsection{Reduction from tuples  to pairs}\label{subsec:tpr} 
For any $n$ and $p\geq 3$ we construct a morphism $$f:V(n,(p-1)(p-2)) \rightarrow V(pn,2)$$ such that the preimage of any $GL_{pn}$- orbit is empty or a $GL_{n}$-orbit. We give the construction only for $p=4$ thereby reducing $6$-tupes to pairs. Then the general case will be obvious. It is convenient to denote a $6$-tuple in $V(n,6)$ by $(x_{31},x_{41},x_{42},y_{13},y_{14}, y_{24})$. Such a tuple is mapped by $f$ to the following pair $(x,y)$ of $4n\times 4n$- block-matrices
 \[
\left [
\begin{array}{cccc}
0&0&0&0\\
1&0&0 & 0\\
x_{31}&1&0&0\\
x_{41}&x_{42}&1&0
\end{array}
\right ] and 
\left [
\begin{array}{cccc}
0&1&y_{13}&y_{14}\\
0&0&1 & y_{24}\\
0&0&0&1\\
0&0&0&0
\end{array}
\right ].
\]
 
 It is a lengthy calculation with block-matrices to verify that the preimage of an orbit is empty or an orbit. This construction from \cite[3.2]{B6} refines the one given earlier by Brenner  in  \cite{B} .

\subsection{Reduction from pairs to commuting nilpotent pairs }\label{subsec:GP} 

This surprising reduction is due to Gelfand and Ponomarev in \cite{G}. We define a morphism 
$$f:mod^{n}\,k\langle X,Y \rangle  \rightarrow mod^{4n}\,k\langle X,Y \rangle/I$$ where $I$ is generated by $XY-YX$ and by all monomials of degree $3$.

 Note that the quotient algebra is commutative of dimension $6$  with only one simple module which is of dimension $1$. Nevertheless the classification of all modules implies the classification of all modules over the free algebra in two variables which has  many simples in each dimension.
 
 To define $f$ we associate to a pair $(X,Y)$ of $n\times n$ matrices the  following pair $(A(X,Y),B(X,Y))$ of two $4n\times 4n$-matrices
 \[ A(X,Y)=
\left [
\begin{array}{cccc}
0&0&0&0\\
I_{n}&0&0 & 0\\
0&0&0&0\\
0&I_{n}&X&0
\end{array}
\right ], \quad B(X,Y)=
\left [
\begin{array}{cccc}
0&0&0&0\\
0&0&0&0\\
I_{n}&0&0&0\\
0&X&Y&0
\end{array}
\right ].
\]
 Here $I_{n}$  is the identity matrix.
 
  First one checks that the pair $(A(X,Y),B(X,Y))$ satisfies  all conditions  imposed by $I$.  It remains  to show that $(X,Y)$ and $(X',Y')$ are conjugate under
  $GL_{n}$ iff $(A(X,Y),B(X,Y))$ and  $(A(X',Y'),B(X',Y'))$ are conjugate under $GL_{4n}$. This is again a straightforward calculation with block-matrices.

  In the language of representation theory all four reductions are induced by appropriate functors between the module categories of the involved algebras.
  These functors are called representation embeddings. The interested reader can find many more examples in \cite{B6}. There the proofs of the above reductions are conceptual with only very few  matrix-calculations.

\section{Comments and related results}\label{sec: crr}
\subsection{A similarity test }\label{subsec:simt}
The recent article  \cite{D} contains the nice result that two tuples $\bA,\bB$ of the same size are simultaneous similar if and only if the  ranks of two linear matrix pencils associated to $\bA$ and $\bB$ coincide for all test matrices. The proof  given there relies on the  matricization of homomorphisms between finite- dimensional modules and the  result \cite[prop 5.2]{Sm} in representation theory. These two things are already contained in our much earlier articles \cite{F1,B3}.

Furthermore  section 7 of \cite{D} refers to \cite{Br} for a deterministic polynomial time algorithm to decide whether two modules are isomorphic or not.  Again there is the much earlier article \cite{B2} to decide this question.2
\subsection{Representations of tame quivers}
For representations of tame quivers - whence in particular for matrix pencils - there is a finite separating stratification by smooth subvarieties with smooth rational geometric quotients just as in our subsections \ref{subsec:ref} and \ref{subsec:rnf}.  This is proved in \cite{B5}.
  
\subsection{Wildness of pairs of matrices}\label{subsec:wild}
Mumford  writes in the second edition of  the famous book \cite{M} on geometric invariant theory on page 167 in a footnote: 

``The classification of two
non-commuting linear maps $f,g:V \rightarrow V\textrm{ mod }GL(V)$ is sometimes referred to as an `impossible' problem. It is not clear to me why this is said. If $dim \,V=n$ then for each $n$ there might be a finite number of explicitely describable algebraic varieties $W_{i}^{n}$ whose points are in natural $1-1$ correspondence with suitable `strata' $S_{i}^{n}$ in the full set $S_{n}$ of pairs $(f,g);k^{n} \rightarrow k^{n}$ mod $GL(n)$.''

It is clear that our stratification is not the solution to this problem because neither the strata  nor the quotients are explicitly described - or are they?

At the end of this article we indicate by four remarks that the above problem is at least very hard if not impossible.
\begin{itemize}
\item Already for pairs of complex ${2\times 2}$- matrices there  exist no normal forms as shown in  \cite{B7} answering a question of Kontsevich in the negative.
\item For pairs of complex $n \times n$ - matrices there is a geometric quotient for the set $S_{n}$ corresponding to simple modules ( see \cite{P}). One conjectures that the quotients are rational, but this is only proven for $n\leq 4$ in \cite{F}.
\item For pairs of $3 \times 3$ matrices one stumbles on  an example of  a smooth variety with a free $PGL(3)$-action having no geometric quotient in \cite[6.3]{B4}. This example is much more elementary than the counter-examples on pages 83-86 of \cite{M}. 
\item Using the undecidability of the word problem for groups it is shown in \cite{Ba} that the theory of pairs of matrices is undecidable in a certain first-order theory.
\end{itemize} 
\subsection{Comments on the papers \cite{F1} and \cite{B3}}\label{subsec:comF1}
In the paper \cite{F1} the second author observed that if $\bA\approx\bB$ then equality  \eqref{eq1.4} holds.  Furthermore,  he showed that the morphism $\varphi$ in Theorem \ref{thm1} maps a $GL_{n}$-orbit into an $H$-orbit. The first named author proved in \cite{B3} our Theorem \ref{thm3} which implies that the inverse image of an $H$-orbit is empty or a $GL_{n}$-orbit.   
Under the assumption that $k$ is the field of the complex numbers $\C$ the second named author showed in \cite{F1} that
there are a finite number of distinct orbits that statisfy the conditions  $E(r,\bA)=E(r,\bB)$ when  the last equality in \eqref{eq1.4} holds.  Needless to say that for $k=\C$ one can use analysis and replace Zariski topology with the standard topology on $\C^m$.  
\subsection{Belitskii's paper on ``Normal forms in matrix spaces''}
Belitskii in \cite{Be} proposed a  finite algorithm for reducing to a canonical form matrices of an arbitrary system of linear mappings,  which is applicable to the simultaneous similarity of matrices.   Two sets of matrices are simultaneous similar if and only if they have the same canonical form. The algorithm of Belitskii affords the knowledge of the eigenvalues of the involved matrices.  So the field has to be algebraically closed and the canonical form can be determined only 'theoretically' whereas our method involves only rational operations with the entries of the
involved matrices.

\end{document}